\documentclass[Leqno, 12pt]{amsart}

\usepackage[paper=a4paper,left=2.95cm,right=2.95cm,top=3.5cm,bottom=3.55cm]{geometry}
\usepackage{float, graphicx, amssymb, dsfont}
\newtheorem{thm}{Theorem}[section]

\newtheorem{lem}[thm]{Lemma}

\theoremstyle{definition}

\newtheorem{exam}[thm]{Example}

\theoremstyle{remark}

\newtheorem*{ack*}{Acknowledgment}

\numberwithin{equation}{section}
\numberwithin{figure}{section}
\newcommand{\cA}{\mathcal{A}}       % alphabet
\newcommand{\cF}{\mathcal{F}}       % forbidden words
\newcommand{\setN}{\mathbb{N}}

\newcommand{\setZ}{\mathbb{Z}}

\newcommand{\A}{\mathcal{A}}    %alphabet
\newcommand{\B}{\mathcal{B}}    %block
\newcommand{\V}{\mathcal{V}}    %vertex sets
\newcommand{\E}{\mathcal{E}}    %edge sets

\newcommand{\AG}{\mathsf{A}}    % adjacency matrix

\newcommand{\X}{\mathsf{X}}

\newcommand{\ifff}{if and only if}

\newcommand{\cto}{constant-to-one}
\newcommand{\fto}{finite-to-one}
\newcommand{\ito}{infinite-to-one}
\newcommand{\rr}{right resolving}
\newcommand{\lr}{left resolving}
\newcommand{\bir}{bi-resolving}
\newcommand{\rc}{right closing}
\newcommand{\lc}{left closing}
\newcommand{\bic}{bi-closing}

\newcommand{\SFT}{shift of finite type}
\newcommand{\SFTs}{shifts of finite type}
\newcommand{\rSFT}{irreducible shift of finite type}
\newcommand{\rSFTs}{irreducible shifts of finite type}
\newcommand{\mSFT}{mixing shift of finite type}
\newcommand{\mSFTs}{mixing shifts of finite type}

\newcommand{\onto}{\xymatrix{\ar@{>>}[r]&}}
\newcommand{\da}[4]{\xymatrix{#1 \ar@<.5ex>[r]^{#2} \ar@<-.5ex>[r]_{#3} & #4}}

\newcommand{\PhiV}{\Phi_\V}
\newcommand{\PhiVi}{\Phi_\V^{-1}}
\newcommand{\PhiE}{\Phi_\E}
\newcommand{\PhiEi}{\Phi_\E^{-1}}

\newcommand{\bPhiV}{{\bar \Phi}_\V}
\newcommand{\bPhiVi}{{\bar \Phi}_\V^{-1}}

\newcommand{\tPhi}{{\tilde \Phi}}

\newcommand{\bicv}{bi-covering}
\newcommand{\rcv}{right covering}
\newcommand{\lcv}{left covering}

\newcommand{\tr}{\mathsf{T}}

\begin{document}

\title[Bi-resolving homomorphisms and bi-closing codes]{Bi-resolving graph homomorphisms and \\ extensions of bi-closing codes}
\author{Uijin Jung}
\address{Department of Mathematical Sciences \\
	Korea Advanced Institute of Science and Technology \\
	Daejeon 305-701 \\
	South Korea}
\email{uijin@kaist.ac.kr}
\dedicatory{Dedicated to the memory of Ki Hang Kim.}
%\urladdr{http://niz.kaist.ac.kr}

\author{In-je Lee}
\email{ijlee@kaist.ac.kr}


\maketitle

\begin{abstract}
    Given two graphs $G$ and $H$, there is a \bir\ (or \bicv) graph homomorphism from $G$ to $H$ \ifff\ their adjacency matrices satisfy certain matrix relations.
    We investigate the \bicv\ extensions of \bir\ homomorphisms and give several sufficient conditions for a \bir\ homomorphism to have a \bicv\ extension with an irreducible domain.
    Using these results, we prove that a \bic\ code between subshifts can be extended to an $n$-to-1 code between \rSFTs\ for all large $n$.
\end{abstract}

\section{Introduction}
Resolving homomorphisms arose independently under different names in different fields of mathematics. These homomorphisms were introduced in the field of symbolic dynamics to solve the finite equivalence problem \cite{AdlGW, AdlM} and they form a fundamental class of finite-to-one codes between subshifts. In particular, all the known general constructions of \fto\ factor codes between \rSFTs\ with equal entropy use resolving codes \cite{Ash93, Boy08, LM}. Covering homomorphisms, resolving ones with the lifting property, are closely related to the graph divisors and equitable partitions in the theory of spectra of graphs (\cite{DenSW} and its references). They also appear in the categorical approach of graph fibration with the name of fibrations and opfibrations \cite{BolV}, and play a significant role in the theory of graph embeddings as ``voltage graphs" \cite{GroT77, GroT}.

\vspace{0.1cm}

This paper is an attempt to investigate the existence and the extension of \bir\ homomorphisms, i.e., both left and right resolving ones. They have more rigid structure than left or right resolving ones \cite{Nas}. Even if two graphs $G$ and $H$ admit a \lcv\ homomorphism and a \rcv\ homomorphism between them, they need not admit a \bicv\ one. We show that there is a \bir\ (resp. \bicv) homomorphism from a graph $G$ to another graph $H$ \ifff\ there is a subamalgamation matrix $S$ such that $ \AG_G S \leq S \AG_H$ and $S^\tr \AG_G \leq \AG_H S^\tr$ (resp. $ \AG_G S = S \AG_H$ and $S^\tr \AG_G = \AG_H S^\tr$), where $\AG_G$ and $\AG_H$ are the adjacency matrices of $G$ and $H$, respectively (see Theorems \ref{thm:bicovering_iff} and \ref{thm:biresolving_iff}).
These results can be considered as an analogue of the well-known description that there is a \rr\ homomorphism from $G$ to $H$ \ifff\ there is a subamalgamation matrix $S$ such that $ \AG_G S \leq S \AG_H$ (e.g., in \cite{LM}). We also investigate \bicv\ extensions of \bir\ homomorphisms in \S 3. Every \bir\ homomorphism $\Phi : G \to H$ can be extended to a \bicv\ homomorphism $\tilde \Phi : \tilde G \to H$ by enlarging the domain. We present sufficient conditions for $\tilde G$ to be irreducible when $H$ is irreducible (see Theorem \ref{thm:graph_extension}).

\vspace{0.1cm}

There has been considerable attention on extending sliding block codes in symbolic dynamics. One can consider the following extension problem: Given a code $\phi : X \to Y$ and $\tilde X \supset X$ with certain properties, extend $\phi$ to a factor code from $\tilde X$ onto $Y$, respecting the properties. There are several results to this extension problem for \ito\ codes \cite{Boy83, BoyT, Jung2}, for \fto\ closing codes \cite{Ash93}, and for inert automorphisms \cite{KimR91}. In particular, Ashley proved that if $X$ and $Y$ are \mSFTs\ and $Y$ is a \rc\ factor of $X$, then any \rc\ code from a \SFT\ $Z \subsetneq X$ can be extended to a \rc\ code from $X$ to $Y$ \cite{Ash93}. An analogous statement for \bic\ codes is false (see Example \ref{ex:biclosing_extension_fail}), thus we are led to consider the weaker version of the extension problem: Given a code $\phi : X \to Y$ with certain properties, construct an enlarged domain $\tilde X$ and extend $\phi$ to a factor code on $\tilde X$, respecting the properties.

The paper \cite{AshMPT} provides many results to this weaker version of the extension problem. One of them concerns \bic\ codes: If $\phi : X \to Y$ is a \bic\ code between \rSFTs, then there are an \rSFT\ $\tilde X$ and a \bic\ extension $\tilde \phi : \tilde X \to Y$ of $\phi$. In \S 4 we extend this result as follows: Given a \bic\ code from a subshift $X$ to an \rSFT\ $Y$ with $h(X) < h(Y)$, for all large $n$ there exist an \rSFT\ $\tilde X$ and an $n$-to-1 (hence \bic) extension from $\tilde X$ onto $Y$. If $X$ is of finite type, then this result holds for every $n$ greater than the maximum number of $\phi$-preimages (see Theorem \ref{thm:biclosing_extension}). This is related to the result in \cite{KimR97}, which says that for a \mSFT\ $X$, there is a family of \mSFTs\  each of which is a \cto\ extension of $X$. (In \cite{KimR97}, each extension is a skew-product of $X$ with a group of the form $\setZ/p\setZ$).

\section{Background}

In this section, we recall some terminology and elementary results. For further details, see \cite{LM}.
A (\emph{directed}) \emph{graph} $G$ is defined to be a pair $(\V, \E)$, where $\V = \V(G)$ is a finite set of vertices and $\E = \E(G)$ is a finite set of edges. We call $G$ \emph{irreducible} if for each pair $(I, J)$ of vertices there exists a path from $I$ to $J$. A graph is \emph{weakly connected} if its underlying graph is connected (i.e., it is possible to reach any vertex starting from any other vertex by traversing edges in some direction). An \emph{irreducible component} of a graph is a maximal irreducible subgraph.

\vspace{0.1cm}

Let $G$ and $H$ be graphs. A (\emph{graph}) \emph{homomorphism} from $G$ into $H$ is a pair $\Phi = (\PhiV, \PhiE)$ of mappings $\PhiV : \V(G) \to \V(H)$ and $\PhiE : \E(G) \to \E(H)$ which respect adjacency. The edge map $\PhiE$ naturally extends to paths. %We write it simply as $\Phi : G \to H$.
We say $\Phi : G \to H$ is \emph{\rr} (resp. \emph{\rcv}) if $\PhiE|_{\E_I(G)} : \E_I(G) \to \E_{\Phi_\V(I)}(H)$ is injective (resp. bijective) for each $I \in \V(G)$, where $\E_I(G)$ is the set of all edges starting from $I$. Similarly, \emph{\lr} and \emph{\lcv} homomorphisms can be defined. If $\Phi$ is both left and right resolving (resp. covering) then it is called \emph{\bir} (resp. \emph{\bicv}).
If $\Phi : G \to H$ is \bicv, then for each path $\pi$ in $H$ the paths in $\PhiE^{-1}(\pi)$ are \emph{mutually separated}, i.e., they do not share a vertex at the same time. %If $H$ is weakly connected and $d = |\V(G)|/|\V(H)|$, then there are exactly $d$ paths in $\PhiE^{-1}(\pi)$. Moreover if $H$ is irreducible, then $G$ must be a disjoint union of irreducible components.

A 0-1 matrix is called a \emph{subamalgamation matrix} if it has exactly one 1 in each row. An \emph{amalgamation matrix} is a subamalgamation matrix which has at least one 1 in each column. For two graphs $G$ and $H$, any subamalgamation matrix $S$ indexed by $\V(G) \times \V(H)$ uniquely determines a vertex mapping $\PhiV$ (and vice versa) by letting $\PhiV(i)=I$ if and only if $S_{i,I}=1$.

\vspace{0.1cm}

In \S 4, we apply the results on graph homomorphisms to obtain certain results in symbolic dynamics. We assume some familiarity with symbolic dynamics. See \cite{Kit, LM} for more on symbolic dynamics.

\section{Existence and extension of bi-resolving homomorphisms}

In this section, we investigate the existence and the extension of \bir\ homomorphisms. We first show that a known necessary condition for the existence of a bi-resolving (resp. bi-covering) homomorphism is also sufficient.

\begin{thm} \label{thm:bicovering_iff}
    Let $G$ and $H$ be graphs. Then there exists a \bicv\ homomorphism from $G$ to $H$ \ifff\ there exists a subamalgamation matrix $S$ with $\AG_G S = S \AG_H \text{ and } S^\tr \AG_G = \AG_H S^\tr.$
\end{thm}
\begin{proof}
    The `only if' part is well known \cite[\S 8.2]{LM}. We will show the converse.

    Let $\AG_G = ( a_{i,j} ), \AG_H = ( b_{I,J} )$ and $S = ( S_{i,I})$. Let $\PhiV : \V(G) \to \V(H)$ be the vertex mapping induced by $S$, i.e., for each $i \in \V(G)$, $\PhiV(i)$ is a unique vertex $I \in \V(H)$ with $S_{i,I} = 1$. Let $\V_I = \PhiVi(I)$ for $I \in \V(H)$. For $I, J \in \V(H)$ with $\V_I, \V_J$ nonempty, let $A_{I,J} = (a_{i,j})_{i \in \V_I, j \in \V_J}$ be the (rectangular) submatrix of $\AG_G$, and let $G_{I,J}$ be the subgraph of $G$ determined by $A_{I,J}$, i.e., $G_{I,J}$ has the vertex set $\V_I \cup \V_J$ and its edge set, say $\E_{I,J}$, which is the set of edges going from a vertex in $\V_I$ to a vertex in $\V_J$. Since $\AG_G S = S \AG_H$, it follows that $\sum_{j \in \V_J} a_{i,j} = b_{I,J}$ for $I, J \in \V(H)$ and $i \in \V_I$, which implies that if $\V_I \neq \emptyset$ and $\V_J \neq \emptyset$, then every row sum of the matrix $A_{I,J}$ is equal to $b_{I,J}$, and that if $\V_I \neq \emptyset$ and $\V_J = \emptyset$, then $b_{I,J} = 0$. Similarly, since $S^\tr \AG_G = \AG_H S^\tr$, we see that if $\V_I \neq \emptyset$ and $\V_J \neq \emptyset$, then every column sum of the matrix $A_{I,J}$ is equal to $b_{I,J}$, and that if $\V_I = \emptyset$ and $\V_J \neq \emptyset$, then $b_{I,J} = 0$. Therefore, if $\V_I \neq \emptyset, \V_J \neq \emptyset$ and $b_{I,J} \neq 0$, then $|\V_I| = |\V_J|$, so that $A_{I,J}$ is a nonnegative integral square matrix with every row and column sum equal to $b_{I,J}$. It is well known that a nonnegative integral square matrix with every row and column sum equal to $R$ is the sum of $R$ permutation matrices (e.g. \cite[\S 5]{LinW}). Therefore $A_{I,J}$ is the sum of $b_{I,J}$ permutation matrices, so that $\E_{I,J}$ is partitioned into disjoint $b_{I,J}$ subsets each of which consists of \emph{vertex-separated} $|\V_I|$ edges (i.e. every distinct two of the $|\V_I|$ edges go neither from the same vertex nor to the same vertex). Hence we can define a graph homomorphism $\Phi_{I,J} : G_{I,J} \to H$ which sends all edges in every one of the $b_{I,J}$ subsets to some one of the $b_{I,J}$ edges going from $I$ to $J$ in $H$ so that the edges in distinct subsets may be sent to distinct edges.

    There exists a graph homomorphism $\Phi : G \to H$ which is an extension of $\Phi_{I,J}$ for all $I, J \in \V(H)$ with $b_{I,J} \neq 0$. It follows that $\Phi$ is \bicv.
\end{proof}

\begin{thm} \label{thm:biresolving_iff}
    Let $G$ and $H$ be graphs. Then there exists a bi-resolving homomorphism from $G$ to $H$ \ifff\ there exists a subamalgamation matrix $S$ with $\AG_G S \leq S \AG_H \text{ and } S^\tr \AG_G \leq \AG_H S^\tr.$
\end{thm}
\begin{proof}
    As in Theorem \ref{thm:bicovering_iff}, the `only if' part is well known \cite[\S 8.2]{LM}. We will show the converse. Notation being the same as in the proof above, let $d = \max \{ |\V_I| : I \in \V(H) \}$. Let $\bar G$ and $\bPhiV : \V(\bar G) \to \V(H)$ be the extensions of $G$ and $\Phi_\V$, respectively, such that $\bar G$ is obtained from $G$ by adding new $d - |\V_I|$ (isolated) vertices, which are mapped to $I$ by $\bPhiV$, for every $I \in \V(H)$. Then $|\V(\bar G)| = d |\V(H)|$. If we define $\bar S = ({\bar S}_{i,I})$ to be the $|\V(\bar G)| \times |\V(H)|$ $0$-$1$ matrix such that $\bPhiV(i) = I$ \ifff\ ${\bar S}_{i,I} = 1$, then $\bar S$ is an amalgamation matrix with $\AG_{\bar G} \bar S \leq \bar S \AG_H$ and ${\bar S}^\tr \AG_{\bar G} \leq \AG_H {\bar S}^\tr$.

    Let $\bar \V_I = \bPhiVi(I)$ for $I \in \V(H)$. Let $\bar A_{I,J} = (\bar a_{i,j})_{i \in \bar \V_I, j \in \bar \V_J}$ be the $d \times d$ submatrix of $\AG_{\bar G} = (\bar a_{i,j})$ for $I, J \in \V(H)$. Then since $\AG_{\bar G} \bar S \leq \bar S \AG_H$, it follows that $\sum_{j \in \bar \V_J} \bar a_{i,j} \leq b_{I,J}$ for $I, J \in \V(H)$ and $i \in \bar \V_I$, so that each row sum of $\bar A_{I,J}$ is not greater than $b_{I,J}$ for each $I, J \in \V(H)$. Similarly, since ${\bar S}^\tr \AG_{\bar G} \leq \AG_H {\bar S}^\tr$, each column sum of $\bar A_{I,J}$ is not greater than $b_{I,J}$ for each $I, J \in \V(H)$. We can obtain from $\bar A_{I,J}$ a nonnegative integral $d \times d$ matrix $\hat A_{I,J}$ with every row and column sum equal to $b_{I,J}$ by adding a necessary number of ``1"s to the components of $\bar A_{I,J}$ (i.e. by adding a necessary number of new edges going from $\bar \V_I$ to $\bar \V_J$). For generally, if $A$ is a nonnegative integral $d \times d$ matrix with every row and column sum not greater than $b$ and with the total sum of components equal to $t$, then $bd - t$ times additions of ``1" to an appropriate component each time, give a matrix $\hat A$ with every row and column sum equal to $b$. (This is straightforwardly proved by induction on $bd - t$.) Let $\hat G_{I,J}$ be the graph determined by $\hat A_{I,J}$. There exists a minimal extension $\hat G$ of $G$ such that $\hat G_{I,J}$ is a subgraph of $\hat G$ for all $I, J \in \V(H)$. Since $\hat A_{I,J}$ has every row and column sum equal to $b_{I,J}$ for all $I, J \in \V(H)$, it follows from the proof of Theorem \ref{thm:bicovering_iff} that there exists a \bicv\ homomorphism $\hat \Phi : \hat G \to H$. The restriction of $\hat \Phi$ on $G$ is a desired bi-resolving homomorphism.
\end{proof}

\vspace{0.1cm}

An \emph{extension} of a homomorphism $\Phi : G \to H$ is a homomorphism $\tilde \Phi : \tilde G \to H$ such that $\tilde G$ is a graph containing $G$ and $\tilde \Phi|_{G} = \Phi$. In the remainder of this section, we investigate \bicv\ extensions of \bir\ homomorphisms. In what follows, the \emph{degree} of a homomorphism $\Phi : G \to H$, denoted by $\deg \Phi$, is the maximum number of preimages of vertices in $H$ under $\PhiV$. For a graph $G$, denote by $\lambda_G$ the spectral radius of its adjacency matrix $\AG_G$.

\begin{thm}\label{thm:graph_extension}
    Let $\Phi : G \to H$ be a bi-resolving homomorphism with $H$ irreducible. Let $d = \deg \Phi$.
    \begin{enumerate}
        \item If $G$ is weakly connected, then there exists a \bicv\ extension $\tilde \Phi : \tilde G \to H$ of $\Phi$ with $\tilde G$ irreducible and $\deg \tPhi = d$.
        \item If $\lambda_H > \lambda_G$ and $n > d$, then there exists a \bicv\ extension $\tilde \Phi : \tilde G \to H$ of $\Phi$ with $\tilde G$ irreducible and $\deg \tPhi = n$.
    \end{enumerate}
\end{thm}
\begin{proof}
    Let $\V_I = \PhiVi(I)$ for $I \in \V(H)$. We may assume that $|\V_I| = d$ by adding new $d - |\V_I|$ (isolated) vertices for every $I \in \V(H)$. Further, we may assume that $n = d + 1$ in (2) by adding new $n - d - 1$ vertices for every $I \in \V(H)$. For $I, J \in \V(H)$, let $B_{I,J}$ the set of edges going from $I$ to $J$.

    Let $I, J \in \V(H)$ with $B_{I,J} \neq \emptyset$. Let $G_{I,J}$ be the subgraph of $G$ whose vertex set is $\V_I \cup \V_J$ and whose edge set, say $\E_{I,J}$, is the set of all edges going from a vertex in $\V_I$ to a vertex in $\V_J$. Note that $\E_{I,J}$ may be empty. Let $\Phi_{I,J} : G_{I,J} \to H$ be the restriction of $\Phi$ on $G_{I,J}$. Since $\Phi$ is \bir, for each $b \in B_{I,J}$, $\PhiE^{-1}(b)$ consists of vertex-separated edges. For each $b \in B_{I,J}$, if $|\PhiE^{-1}(b)| < d$, we can add new $d - |\PhiE^{-1}(b)|$ edges, which will be called the \emph{new} $b$-edges, to the graph $G_{I,J}$ so that the \emph{new} $b$-edges together with the edges in $\PhiE^{-1}(b)$, which we will call the \emph{old} $b$-edges, may be all vertex-separated. Let $\hat G_{I,J}$ be the graph extension of $G_{I,J}$ with the \emph{new} $b$-edges added for all $b \in B_{I,J}$. Let $\hat \Phi_{I,J} : \hat G_{I,J} \to H$ be the extension of $\Phi_{I,J}$ which sends all \emph{old} and \emph{new} $b$-edges to $b$ for all $b \in B_{I,J}$.

    There exists a minimal extension $\hat G$ of $G$ such that $\hat G_{I,J}$ is a subgraph of $\hat G$ for all $I, J \in \V(H)$ with $B_{I,J} \neq \emptyset$. There exists an extension $\hat \Phi : \hat G \to H$ of $\Phi$ whose restriction of $\hat \Phi$ on $\hat G_{I,J}$ is $\hat \Phi_{I,J}$ for all $I, J \in \V(H)$ with $B_{I,J} \neq \emptyset$. Since $\hat \Phi : \hat G \to H$ is \bicv\ and $H$ is irreducible, $\hat G$ is the disjoint union of finitely many irreducible graphs. Therefore if $G$ is weakly connected, then there exists an irreducible component $\tilde G$ such that $G$ is a subgraph of $\tilde G$. Hence the restriction $\tilde \Phi$ of $\hat \Phi$ on $\tilde G$ is a \bicv\ homomorphism desired in (1). (Indeed, one can check that $\tilde G = \hat G$ in this case.)

    Suppose now $\lambda_G < \lambda_H$. Then each irreducible component of $\hat G$ has a \emph{new} $b$-edge for some edge $b$ in $H$. Let $\hat G_1, \cdots, \hat G_m$ be the irreducible components of $\hat G$ and let $e_k$ be a \emph{new} $b_k$-edge in $\hat G_k$ with $b_k$ = $\hat \Phi_\E(e_k)$ for $k = 1, \cdots, m$. Let $\hat G_0$ be a copy of $H$. We assume that for each edge $b$ in $H$, the copy $e_b$ (in $\hat G_0$) of $b$ is a \emph{new} $b$-edge. For $k = 0 , \cdots, m$, we define an irreducible graph $\tilde G_k$ inductively. Let $\tilde G_0 = \hat G_0$. Assuming that $\tilde G_{k-1}$ is an irreducible graph which has a \emph{new} $b$-edge for all edge $b$ in $H$, we define $\tilde G_k$ as follows: let the \emph{new} $b_k$-edge $e_k$ and one of the \emph{new} $b_k$-edges of $\tilde G_{k-1}$ exchange their terminal vertices; then we can merge $\hat G_k$ and $\tilde G_{k-1}$ into one irreducible graph $\tilde G_k$, which has a \emph{new} $b$-edge for all edge $b$ in $H$.

    Let $\tilde G = \tilde G_m$. We have a graph homomorphism $\tilde \Phi : \tilde G \to H$ which sends all \emph{old} and \emph{new} $b$-edges to $b$ for all edges $b$ in $H$. Then $\tilde \Phi$ is a \bicv\ extension of $\Phi$ with $\tilde G$ irreducible and $\deg \tilde \Phi = d + 1 = n$. Hence (2) is proved.
\end{proof}

The proof of Theorem \ref{thm:graph_extension}(1) shows that \emph{every} bi-resolving homomorphism can be extended to a bi-covering one with the same degree by enlarging the domain. We remark that Theorem \ref{thm:graph_extension} also holds if we replace  \emph{irreducible} with \emph{weakly connected}. Note that the assumption $\lambda_H > \lambda_G$ in the theorem is crucial. Indeed, an application of Perron-Frobenius theorem shows that if $H$ is irreducible, $G$ is not irreducible and $\lambda_G = \lambda_H$, then $\tilde G$ cannot be irreducible for any \bicv\ extension $\tPhi : \tilde G \to H$ of $\Phi$.

\begin{exam} \label{ex:equal_degree_is_impossible_in_general}
    Let $G$ and $H$ be graphs as below and $\Phi : G \to H$ a subscript dropping homomorphism. It is easy to check that there is no \bicv\ extension of $\Phi$ with degree 2 and with a weakly connected domain. This example shows that the assumption $n > \deg \Phi$ in Theorem \ref{thm:graph_extension}(2) is crucial.

    \vspace{-0.13cm}
    \begin{figure}[h]
        \center \includegraphics[height=2.8cm]{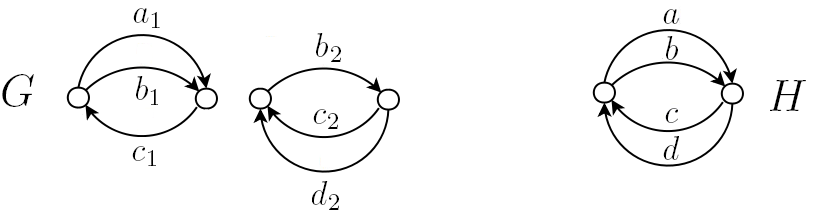}
        \vspace{-0.35cm}
        \caption{}
        \label{fig:same_degree_impossible}
    \end{figure}
\end{exam}

\section{Extension of bi-closing codes}

In this section, we investigate the extension property of bi-closing codes between general shift spaces. We prove that a \bic\ code between subshifts can be extended to an $n$-to-1 code between \rSFTs\ for all large $n$. When the domain is of finite type, we give a lower bound of degrees of extensions in the sense that there is $N \in \setN$ such that the above result holds for every $n \geq N$.

We recall some definitions. A \emph{shift space} (or \emph{subshift}) is a closed shift-invariant subset of a full shift. A subshift is \emph{indecomposable} if it is not the union of two disjoint nonempty subshifts \cite{BK}, and \emph{irreducible} if it has a dense forward orbit. For a subshift $X$, denote by $\B_n(X)$ the set of all words of length $n$ appearing in the points of $X$, and by $h(X)$ the topological entropy of $X$. A \emph{code} is a continuous shift-commuting map between shift spaces. A code is \emph{\rc} (resp. \emph{\lc}) if it never collapses two distinct left (resp. right) asymptotic points, and \emph{\bic} if it is both left and right closing. The \emph{edge shift} $\X_G$ is the set of all bi-infinite trips on a graph $G$. Every homomorphism $\Phi : G \to H$ induces the code $\phi : \X_G \to \X_H$ by letting $\phi(x)_i = \Phi(x_i)$. A subshift is called a \emph{shift of finite type} if it is conjugate to an edge shift.

We adopt the ideas from \cite[Lemma 2.4]{BK} to prove the following lemma on recoding.

\begin{lem}\label{lem:code-recoding_to_original_space} %\cite[Lemma 2.4]{BK}
    Let $X$ and $\bar X$ be shift spaces with $X \subset \bar X$ and $\phi : X \to Y$ a conjugacy. Then there exist a shift space $\bar Y \supset Y$ and a conjugacy $\bar \phi : \bar X \to \bar Y$ such that $\bar \phi|_X = \phi$.
\end{lem}
\begin{proof}
    We may assume that $\phi$ is $1$-block and $\phi^{-1}$ has $N \in \setN$ as its  memory and anticipation, and that no $X$-word occurs as a symbol for $Y$. Define an alphabet $\A = \B_1(Y)~ \dot \cup~ ( \B_{2N+1}(\bar X) \setminus \B_{2N+1}(X))$. We will regard $u \in \B_{2N+1}(\bar X) \setminus \B_{2N+1}(X)$ as a symbol in $\A$. Define $\bar \phi : \bar X \to \A^\setZ$ by
    \[
        {\bar \phi (x) }_i =
            \begin{cases}
                \phi(x_i) & \text{if $x_{[-N+i,N+i]} \in \B_{2N+1}(X)$}   \\
                x_{[-N+i,N+i]} & \text{otherwise.}   \\
            \end{cases}
    \]
    and let $\bar Y = \bar \phi(\bar X)$. Note that if $\phi(x) = y$, then each $y_{[-N+i,N+i]}$ determines $x_i$ uniquely. Thus $\bar \phi$ is a conjugacy onto its image $\bar Y$.
\end{proof}

For a graph $G$ and $N \in \setN$, denote by $G^{[N]}$ the \emph{$N$-th higher graph} of $G$ \cite[\S 2.3]{LM}. A graph homomorphism $\Phi : G \to H$ naturally induces the graph homomorphism $\Phi^{[N]} : G^{[N]} \to H^{[N]}$ for each $N$. If $\Phi$ is \bir, then clearly so is $\Phi^{[N]}$. For a shift space $X$ and $N \in \setN$, denote by $X^{[N]}$ the \emph{$N$-th higher block shift} of $X$. Then $X$ is conjugate to $X^{[N]}$ by the conjugacy $\beta_{N,X} : X \to X^{[N]}$ where $\beta_{N,X} (x)_i = x_{[i,i+N-1]}$.
It is known that a code between \rSFTs\ is \cto\ if and only if it is conjugate to a code induced by a \bicv\ homomorphism \cite{Nas}. In this case the number of preimages of each point under the code is equal to the degree of the homomorphism. In what follows, a graph is called \emph{essential} if each vertex has an incoming edge and an outgoing edge.

\begin{lem} \label{lem:code-lowering_degree}
    Let $\Phi : G \to H$ be a \bir\ homomorphism where $G$ and $H$ are essential. Let $d = \max \{ |\phi^{-1}(y)| : y \in \X_H \}$, where $\phi : \X_G \to \X_H$ is the  code induced by $\Phi$. Then there exists $N \in \setN$ such that $\deg \Phi^{[N]} = d$.
\end{lem}
\begin{proof}
    Since $\Phi$ is \bir, the preimages of a given path must be mutually separated. Since $G$ and $H$ are essential, it follows from compactness that there is $n \in \setN$ such that $|\PhiEi(\pi)| \leq d$ for all paths $\pi$ of length greater than $n$. Take $N = n+2$. Then $| ( \Phi^{[N]})_\V ^{-1}(J)| \leq d$ for all $J \in \V(G^{[N]})$. Thus $\deg \Phi^{[N]} \leq d$.

    Suppose $\deg \Phi^{[N]} < d$ and let $\phi^{[N]}$ be the code induced by $\Phi^{[N]}$. Since $\Phi^{[N]}$ is \bir, the number of preimages of a point under $\phi^{[N]}$ must be less than $d$. This contradicts that $\phi$ is conjugate to $\phi^{[N]}$. Thus $\deg \Phi^{[N]} = d$.
\end{proof}

Now we prove the main theorem of this section which says that, in a sense, every bi-closing code sits in a \cto\ code between \rSFTs.

\begin{thm} \label{thm:biclosing_extension}
    Let $X$ be a shift space, $Y$ an \rSFT\ with $h(X) < h(Y)$, and $\phi : X \to Y$ a \bic\ code. Let $d = \max \{ |\phi^{-1}(y)| : y \in Y \}$.
    \begin{enumerate}
        \item
            If $X$ is an indecomposable \SFT, then for all $n \geq d$, there exist an \rSFT\ $\tilde X \supset X$ and an $n$-to-1, onto extension $\tilde \phi : \tilde X \to Y$ of $\phi$.
        \item
            If $X$ is of finite type, then the conclusion of (1) holds if ``for all $n \geq d$'' is replaced by ``for all $n \geq d+1$''.
        \item
            If $X$ is a shift space, then the conclusion of (1) holds if ``for all $n \geq d$'' is replaced by ``for all $n \geq m$ with some $m \geq d$''.
    \end{enumerate}
\end{thm}
\begin{proof}
    (1)(2) Since $X$ and $Y$ are \SFTs\ and $Y$ is irreducible, using the higher block presentations and the recoding construction of \cite[\S 4.3]{Kit}, we know that there exist an essential graph $G$, an irreducible graph $H$, a \bir\ homomorphism $\Psi : G \to H$, a conjugacy $\alpha : X \to \X_{G}$, and a higher block code $\beta : Y \to \X_{H}$ such that $\phi = \beta^{-1} \psi \alpha$, where $\psi$ is the 1-block code induced by $\Psi$. By Lemma \ref{lem:code-lowering_degree}, there exists $N \geq 1$ such that $\deg \Psi^{[N]} = d$. Letting $G_1 = G^{[N]}, H_1 = H^{[N]}$ and $\Psi_1 = \Psi^{[N]}$, we have $\phi = \beta^{-1} \beta^{-1}_{N,\X_H} \psi_1 \beta_{N,\X_G} \alpha$, where $\psi_1$ is the 1-block code induced by $\Psi_1$.

    By Theorem \ref{thm:graph_extension}, there exists a \bicv\ extension $\tilde \Psi_1 : \tilde G_1 \to H_1$ of $\Psi_1$ with $\tilde G_1$ irreducible and $\deg \tilde \Psi_1 = n$. By Lemma \ref{lem:code-recoding_to_original_space}, there exist an \rSFT\ $\tilde X$ with $\tilde X \supset X$ and a conjugacy $\theta : \tilde X \to \X_{\tilde G_1}$ such that $\theta|_{X} = \beta_{N,\X_G} \alpha $. If we let $\tilde \phi = \beta^{-1} \beta^{-1}_{N,\X_H} \tilde \psi_1 \theta$, where $\tilde \psi_1$ is the 1-block code induced by $\tilde \Psi_1$, then $\tilde \phi$ is an extension desired in (1) and (2).

    (3) Let $\cA = \B_1(X)$. Then $X$ is a subshift over the alphabet $\cA$. Define $\phi$ by an $M$-block map $\Phi : \B_M(X) \to \B_1(Y)$ with memory $m$ and anticipation $a$ with $m + a + 1 = M$ and let $N$ be a number such that $Y^{[N]}$ is an edge shift. For $k \geq M + N$, let $X_k$ be the \SFT\ defined by the set $\cF_k = \A^k \setminus \B_k(X)$ of forbidden blocks (i.e., $X_k$ is the $k$-step Markov approximation of $X$). Then we can define the code $\phi_k : X_k \to Y$ with the $M$-block map $\Phi$ with memory $m$ and anticipation $a$ (note that $\phi_k(X_k) \subset Y$). Clearly $X_{k+1} \subset X_k$ for all $k$ and $X = \bigcap_k X_k$. Since $h(X) < h(Y)$, for all large $k$ we have $h(X_k) < h(Y)$. Since $\phi$ is \bic, it follows by a standard compactness argument that $\phi_k$ is \bic\ for all large $k$. Therefore we can apply (2) for $\phi_k$ with sufficiently large $k$ to prove (3).
\end{proof}

This theorem may be viewed as an another aspect of the extension result in \cite[\S 4]{AshMPT}. Indeed it gives more information except for the closing delay which is defined only for 1-block codes on 1-step \SFTs. Note that in Theorem \ref{thm:biclosing_extension}, if $X$ and $Y$ are mixing then so is $\tilde X$.

%Let $X$ and $Y$ be \mSFTs\ and $Z \subsetneq X$ of finite type. As stated in \S 1, if there is a \rc\ factor code from $X$ to $Y$, then any \rc\ code from $Z$ to $Y$ can be extended to a \rc\ factor code on $X$ \cite{Ash93}. The following example shows that an analogous statement for \bic\ codes is not true.

Our last example shows that we cannot improve the extension theorem of Ashley \cite{Ash93} by replacing \rc\ with \bic.

\begin{exam} \label{ex:biclosing_extension_fail}
    Let $X = Y = \X_A$ where $A = \left( \begin{smallmatrix} 1 & 2 \\ 1 & 0 \end{smallmatrix} \right)$ and $Z$ be a periodic orbit of $X$ of length greater than 1. Let $\phi : Z \to Y$ be the code that maps every point of $Z$ to the unique fixed point of $Y$. Clearly $\phi$ is \bic. If there is a \bic\ extension $\tilde \phi : X \to Y$ of $\phi$, then it must be constantly $d$-to-1 with $d > 1$. However, since $-1$ is an eigenvalue of $A$, it follows that $X$ only admits endomorphisms of degree one \cite{Tro90}, which is a contradiction. Thus $\phi$ cannot be extended to a \bic\ code from $X$ onto $Y$.
\end{exam}

\begin{ack*}
    This paper was written during the authors' graduate studies. We would like to thank our advisor Sujin Shin for her encouragement and good advice. We thank the referee for thorough suggestions which simplified and clarified many proofs, and Mike Boyle for comments which improved the exposition. This work was supported by the second stage of the Brain Korea 21 Project, The Development Project of Human Resources in Mathematics, KAIST in 2009.
\end{ack*}

\bibliographystyle{amsplain}
\bibliography{_Bib_Biresolving}

\end{document}